\documentclass[11pt]{article}

\usepackage[T1]{fontenc}
\usepackage[utf8]{inputenc}
\usepackage{lmodern}
\usepackage{amsmath,amssymb,amsthm,mathtools}
\usepackage[margin=1in]{geometry}
\usepackage[colorlinks=true,linkcolor=blue,citecolor=blue,urlcolor=blue]{hyperref}

\numberwithin{equation}{section}

\newtheorem{theorem}{Theorem}[section]
\newtheorem{proposition}[theorem]{Proposition}
\newtheorem{corollary}[theorem]{Corollary}
\newtheorem{lemma}[theorem]{Lemma}
\theoremstyle{remark}
\newtheorem{remark}[theorem]{Remark}

\newcommand{\R}{\mathbb{R}}
\newcommand{\N}{\mathbb{N}}
\newcommand{\loc}{\mathrm{loc}}
\newcommand{\crit}{\mathcal{Z}}
\newcommand{\dist}{\operatorname{dist}}
\newcommand{\Dp}{\Delta_p}

\title{Higher Regularity of Homogeneous Gradient Compositions for
$p$-Laplace-Type Equations}
\author{Quoc Hung Nguyen\thanks{Academy of Mathematics and Systems
Science, Chinese Academy of Sciences, Beijing 100190, China.  E-mail:
\texttt{qhnguyen@amss.ac.cn}.}
\and Le Xuan Truong\thanks{University of Economics Ho Chi Minh City
(UEH), Ho Chi Minh, Vietnam.  E-mail:
\texttt{lxuantruong@ueh.edu.vn}.}}
\date{}
\begin{document}
\maketitle

\begin{abstract}
In this paper, we study higher regularity of homogeneous functions of
the gradient of solutions to the inhomogeneous $p$-Laplace equation
$\operatorname{div}(|Du|^{p-2}Du)=f$.  Although a solution need not be
of class $C^2$ across its critical set, its gradient is locally
Hölder continuous.  Suppose that $Du\in C^{0,\alpha}_{\loc}$ with
$\alpha\le 1/(p-1)$, and let $\Phi$ be smooth away from the origin and
positively homogeneous of degree $m$.  We prove that
$\Phi(Du)\in C^k_{\loc}$ whenever $m>k/\alpha$.  Moreover, all its
derivatives of order at most $k$ vanish on the critical set.  The proof
uses the intrinsic scale $r\simeq |Du|^{1/\alpha}$, Schauder estimates
for a normalized uniformly elliptic equation, and an extension lemma
across the critical set.  We also obtain corresponding results for
autonomous anisotropic equations and for elliptic and parabolic
$p$-Laplace systems, under the appropriate Hölder assumption on the
gradient.  Finally, the same argument gives $C^k$ regularity criteria
for high powers of nonnegative solutions to the porous medium
equation.
\end{abstract}



\section{Introduction}
In this paper, we are concerned with weak solutions
$u\in W^{1,p}_{\loc}(\Omega)$ of
\begin{equation}\label{eq:p-poisson-intro}
 \Dp u:=\operatorname{div}\bigl(|Du|^{p-2}Du\bigr)=f
 \qquad\text{in }\Omega.
\end{equation}
Here $\Omega$ is an open subset of $\R^d$ and $1<p<\infty$.
The coefficient matrix obtained by differentiating the flux is
\[
 |Du|^{p-2}\left(I+(p-2)\frac{Du\otimes Du}{|Du|^2}\right).
\]
Thus, even when $f$ is smooth, the equation is degenerate for
$p>2$ and singular for $1<p<2$ at points where $Du=0$.  The
classical results of Evans \cite{Evans1982}, DiBenedetto
\cite{DiBenedetto1983}, Tolksdorf \cite{Tolksdorf1984}, and Lieberman
\cite{Lieberman1988,Lieberman1991} show, under standard assumptions on
the datum, that
\begin{equation}\label{eq:classical-holder}
 Du\in C^{0,\alpha_0}_{\loc}(\Omega)
\end{equation}
for some $\alpha_0=\alpha_0(d,p)\in(0,1)$.  More precise gradient
estimates can be expressed in terms of nonlinear potentials; see
\cite{DuzaarMingione2010,DuzaarMingione2010Lipschitz,
DuzaarMingione2011,Mingione2011Gradient,KuusiMingione2012Universal,
KuusiMingione2014Guide,AvelinKuusiMingione2018,
KuusiMingione2014,KuusiMingione2018}.  For singular $p$-Laplace type
equations with measure data, comparison estimates, universal potential
bounds, pointwise gradient estimates, weighted good-$\lambda$
inequalities, and existence and regularity results were developed in
\cite{NguyenPhuc2019GoodLambda,NguyenPhuc2020Pointwise,
NguyenPhuc2022Measure,NguyenPhuc2023Comparison,
NguyenPhuc2023Universal}.  

In general, local $C^{1,\alpha}$ regularity, for some
$\alpha\in(0,1)$, is the best one can expect for solutions of the
$p$-Laplace equation.  For example, the one-dimensional profile
\[
 u(x)=\frac{p-1}{p}|x_1|^{p/(p-1)}
\]
satisfies $\Delta_p u=1$, but it is not $C^2$ across $\{x_1=0\}$ when
$p>2$.  It is therefore natural to ask whether a sufficiently high
power of the gradient has better regularity.  More precisely, given
$k\in\N$, we ask whether there is an exponent $m_k$ such that
\begin{center}
	$|Du|^{m_k}\in C^k_{\loc}(\Omega)$.
\end{center}
We give an explicit sufficient condition on $m_k$.  We also consider
elliptic and parabolic $p$-Laplace systems and the porous medium
equation.

Away from the critical set
\[
 \crit(u):=\{x\in\Omega:Du(x)=0\},
\]
equation \eqref{eq:p-poisson-intro} is locally uniformly elliptic.
Hence $u$ is smooth on $\Omega\setminus\crit(u)$ when $f$ is smooth.
The main point is to control the derivatives as one approaches
$\crit(u)$.

There is an extensive regularity theory for nonlinear quantities
associated with the equation.  Second-order and two-sided estimates
for the stress field
$|Du|^{p-2}Du$ under weak assumptions on $f$ were developed by Cianchi
and Maz'ya \cite{CianchiMazya2018,CianchiMazya2019}; related Sobolev and
BMO estimates appear in \cite{DieningKaplickySchwarzacher2012,
HaaralaSarsa2022,Lou2008}.  Global gradient estimates and boundary
results require additional assumptions on the datum and on the domain;
see, among others, \cite{CianchiMazya2011,CianchiMazya2014a,
CianchiMazya2014b,CianchiMazya2015,MontoroMugliaSciunzi2025}.  Hessian
estimates and regularity for related nonlinear fields are studied in
\cite{Cellina2017,DamascelliSciunzi2004,DongPengZhangZhou2020,
Sciunzi2007,Sciunzi2014}.  These results show, in particular, that a
suitable nonlinear function of $Du$ may have better weak regularity
than $D^2u$ itself.  A weighted second-order counterpart for degenerate
$p$-Laplace type equations is given in
\cite{AntoniniCiraoloPagliarin2025}.

When $p$ is close to $2$, perturbative Calder\'on--Zygmund arguments
give stronger differentiability.  Results near uniform ellipticity
were obtained in \cite{MercuriRieySciunzi2016,GuarnottaMosconi2023}.
Baratta--Sciunzi--Vuono \cite{BarattaSciunziVuono2025} prove
third-order estimates and regularity of the stress field, while
Iandoli--Vuono \cite{IandoliVuono2025} study the local behavior of
second derivatives.  Most recently, Iandoli--Spadaro--Vuono
\cite{IandoliSpadaroVuono2026} established, for $p$ sufficiently close
to $2$, higher-order Sobolev regularity for $|Du|^{s-2}Du$ and
$|Du|^{s-2}D^2u$, as well as weighted local $L^\infty$ bounds for their
higher derivatives.  Their proof uses regularization, differentiated
energy estimates, Calder\'on--Zygmund inequalities, multivariate
Fa\`a di Bruno formulas \cite{ConstantineSavits1996}, and Moser
iteration \cite{Moser1961,Serrin1964}.

Our result has a different form.  We consider a smooth homogeneous
function $\Phi$ and study the classical regularity of $\Phi(Du)$
across the whole critical set.  The argument applies to every
$1<p<\infty$ once a gradient Hölder exponent is available.  It is
based on a pointwise normalization and local Schauder estimates, and
does not require $p$ to be close to $2$.  The required degree of
homogeneity depends on the exponent in \eqref{eq:classical-holder}.

We also recall some related results for anisotropic equations and
systems.
Interior and global second-order estimates for anisotropic equations
are proved in \cite{AntoniniCiraoloFarina2023,
AntoniniCianchiCiraoloFarinaMazya2025,BarattaMugliaVuono2025,
CastorinaRieySciunzi2019}; recent Lipschitz and $C^{1,\beta}$ results
include \cite{AntoniniCianchi2025,Antonini2026}.  For systems, see
\cite{BalciCianchiDieningMazya2022,Miskiewicz2018,
MontoroMugliaSciunziVuono2025,SciunziSpadaroVuono2025}; the
foundational full-regularity result for systems with Uhlenbeck
structure is \cite{Uhlenbeck1977}.
For the parabolic $p$-Laplace system, interior Hölder continuity of the
spatial gradient goes back to DiBenedetto--Friedman
\cite{DiBenedettoFriedman1985}; see also
\cite{Boegelein2015,BoegeleinDuzaarLiaoScheven2022,
BoegeleinDuzaarGianazzaLiaoScheven2025,
KuusiMingione2013Parabolic,KuusiMingione2014Wolff} for quantitative,
potential-theoretic, and more general developments.
Nonautonomous and nonuniformly elliptic problems involve additional
difficulties \cite{DeFilippisMingione2020,DeFilippisMingione2023,
DeFilippisMingione2023Schauder}.  In the anisotropic case considered
below, we keep the operator autonomous and homogeneous, since these
properties are preserved by the normalization.

We next describe the main idea of the proof.  Consider first the model
quantity
\[
 W_m:=|Du|^m.
\]
If $z\in\crit(u)$, then \eqref{eq:classical-holder} gives
\[
 W_m(x)\lesssim |x-z|^{m\alpha_0}.
\]
This immediately suggests differentiability at $z$ when
$m\alpha_0>1$, but it does not by itself control the derivative of
$W_m$ at nearby noncritical points.  That control comes from the
equation.  At a point $x$ with $q=|Du(x)|>0$, the natural radius on
which $|Du|$ remains comparable to $q$ is
\begin{equation}\label{eq:intrinsic-scale-intro}
 r\simeq q^{1/\alpha}.
\end{equation}
After normalization on this scale, the equation becomes uniformly
elliptic on a fixed ball.  Schauder estimates then give, schematically,
\begin{equation}\label{eq:schematic-estimates}
 |D^\ell u(x)|\lesssim
 q^{1-(\ell-1)/\alpha},
 \qquad
 |D^j W_m(x)|\lesssim q^{m-j/\alpha}.
\end{equation}
The second estimate tends to zero as $x$ approaches the critical set
whenever $m>j/\alpha$.  An extension lemma then gives the derivatives
on $\crit(u)$ and yields $C^k$ regularity for $m>k/\alpha$.  Notice
that the H\"older continuity of $Du$ alone only controls the values of
$W_m$ near the critical set.  The equation is used to estimate its
derivatives at nearby noncritical points.

Our main result applies to every smooth positively homogeneous map
$\Phi$ and gives the estimate
$|D^j(\Phi(Du))|\lesssim |Du|^{m-j/\alpha}$.  We also prove analogous
results for autonomous anisotropic operators and for vectorial
$p$-Laplace systems.  In the parabolic case we keep track separately
of spatial and temporal derivatives.  The last section treats powers
of nonnegative solutions to the porous medium equation.  The proofs
use classical H\"older and Schauder estimates; see
\cite{GilbargTrudinger,GiaquintaMartinazzi2012}.

We now state the main result.  A map
$\Phi:\R^d\setminus\{0\}\to\R^N$ is positively homogeneous of degree
$m$ if
\begin{equation}\label{eq:homogeneity}
 \Phi(t\xi)=t^m\Phi(\xi) \qquad (t>0,\ \xi\ne0).
\end{equation}

\begin{theorem}\label{thm:main}
Let $k\in\N$, let $f\in C^\infty(\Omega)$, and let
$u\in W^{1,p}_{\loc}(\Omega)$ solve \eqref{eq:p-poisson-intro}.
Suppose that \eqref{eq:classical-holder} holds, and fix
$\alpha\in(0,1)$ such that
\begin{equation}\label{eq:alpha-choice}
	\alpha\le \min\left\{\alpha_0,\frac1{p-1}\right\}.
\end{equation}
Let $\Phi\in C^\infty(\R^d\setminus\{0\};\R^N)$ be positively
homogeneous of degree $m$.  If $m>k/\alpha$, extend $\Phi$
continuously to the origin by setting $\Phi(0)=0$.  Then the extended
map satisfies
\[
 	\Phi(Du)\in C^k_{\loc}(\Omega;\R^N).
\]
Moreover, for every $0\le j\le k$,
\begin{equation}\label{eq:vanishing-on-critical-set}
	 D^j\bigl(\Phi(Du)\bigr)=0 \quad\text{on } \quad \crit(u).
\end{equation}
In addition, for every $K\Subset\Omega$, there are $q_0>0$ and $C>0$
such that
\begin{equation}\label{eq:composition-quantitative-intro}
 \left|D^j\bigl(\Phi(Du)\bigr)(x)\right| \le C|Du(x)|^{m-j/\alpha}
\end{equation}
whenever $x\in K\setminus\crit(u)$, $|Du(x)|\le q_0$, and $0\le j\le k$.
\end{theorem}

\begin{remark}
The exponent $\alpha_0$ is supplied by classical gradient regularity
and depends on the structural parameters $d$ and $p$.  The theorem
uses any $0<\alpha\le\min\{\alpha_0,1/(p-1)\}$; no optimality of
$\alpha_0$ is assumed.  Thus the sufficient degree depends on
$d,p$, and $k$ through the chosen exponent, whereas local norms of
$u$ and $f$ enter only the constants in the quantitative estimates.
\end{remark}

Taking $\Phi(\xi)=|\xi|^m$, we obtain the following consequence.

\begin{corollary}[Powers of the gradient]
\label{cor:power}
Under the hypotheses of Theorem \ref{thm:main},
\begin{equation}\label{eq:power-ck}
 |Du|^m\in C^k_{\loc}(\Omega)
 \qquad\text{if}\qquad
 m>\frac{k}{\alpha}.
\end{equation}
In particular,
\begin{equation}\label{eq:power-c1}
 |Du|^m\in C^1_{\loc}(\Omega)
 \qquad\text{for every}\qquad
 m>\max\left\{\frac1{\alpha_0},p-1\right\}.
\end{equation}
Its derivative is
\begin{equation}\label{eq:power-derivative}
 D(|Du|^m)=
 \begin{cases}
 m|Du|^{m-2}D^2u\,Du,&Du\ne0,\\[3pt]
 0,&Du=0.
 \end{cases}
\end{equation}
\end{corollary}

\begin{remark}[Dependence of the threshold]
\label{rem:threshold-dependence}
The sufficient degree $m>k/\alpha$ depends on the available gradient
H\"older exponent and on the desired differentiability order, but not
on the size of $f$.  Local norms of $f$ and of $Du$ enter the constants
$q_0$ and $C$ in \eqref{eq:composition-quantitative-intro}.  Thus a
larger datum changes the quantitative neighborhood on which the
intrinsic estimate is applied, not the homogeneity threshold itself.
\end{remark}

Theorems \ref{thm:system} and \ref{thm:parabolic-system} give the
corresponding results for elliptic and parabolic $p$-Laplace systems.
For the parabolic problem, spatial and temporal derivatives have
different homogeneity costs.  Proposition \ref{prop:porous-power}
treats powers of nonnegative solutions to the porous medium equation.

\begin{remark}[On the size of the critical set]
Theorem \ref{thm:main} does not give an estimate for the measure or
Hausdorff dimension of $\crit(u)$.  Such an estimate is not possible
under the present assumptions.  Indeed, a constant solution with
$f\equiv0$ has $\crit(u)=\Omega$.  Additional nondegeneracy or
unique-continuation assumptions would be needed to study the size of
the critical set.
\end{remark}

The restriction involving $p-1$ is unavoidable for the general
inhomogeneous equation.  Indeed, the one-dimensional constant-source
solution
\begin{equation}\label{eq:model-intro}
 u(x)=\frac{p-1}{p}|x_1|^{p/(p-1)}
\end{equation}
satisfies $\Dp u=1$, whereas
$|Du|^m=|x_1|^{m/(p-1)}$.  Thus the strict condition $m>p-1$ is
necessary for a universal $C^1$ statement.

For $p>2$ in the plane, the optimal estimate
$u\in C^{1,1/(p-1)}_{\loc}$ for bounded data was proved by
Araújo--Teixeira--Urbano \cite{AraujoTeixeiraUrbano2017}, building on
the planar $p$-harmonic theory of Iwaniec--Manfredi
\cite{IwaniecManfredi1989}.  In this setting \eqref{eq:power-c1}
reduces to the sharp range $m>p-1$.  The corresponding optimal
higher-dimensional regularity problem is substantially more delicate;
see \cite{AraujoTeixeiraUrbano2018}.

The argument is not tied to rotational invariance.  In Section
\ref{sec:anisotropic} we prove the analogous statement for autonomous
anisotropic operators $\operatorname{div}A(Du)$ for which $A$ is
$(p-1)$-homogeneous and elliptic on the unit sphere.

The paper is organized as follows.  Section \ref{sec:extension}
contains the extension lemma.  Sections \ref{sec:intrinsic} and
\ref{sec:composition} establish the scalar intrinsic estimates and
prove Theorem \ref{thm:main}.  Section \ref{sec:anisotropic} treats
autonomous anisotropic equations, Section \ref{sec:systems} treats the
elliptic vectorial system, and Section \ref{sec:parabolic-system}
develops the parabolic counterpart.  Section \ref{sec:porous-medium}
records the corresponding porous-medium application.

\section{An extension lemma across a zero set}
\label{sec:extension}

We first record an elementary lemma that separates the extension
argument from the PDE estimates.  The statement is given for scalar
functions; it applies componentwise to vector-valued maps.

\begin{lemma}[Flat extension]
\label{lem:extension}
Let $U\subset\R^d$ be open, let $Z\subset U$ be closed, let $k\in\N$,
and let $\mu>k$.  Suppose that $F\in C^\infty(U\setminus Z)\cap C(U)$
and $F=0$ on $Z$.  Assume that for every $K\Subset U$ there is a
constant $C_K$ and a radius $\rho_K>0$ such that
\begin{equation}\label{eq:extension-assumption}
 |D^jF(x)|\le C_K\dist(x,Z)^{\mu-j},
 \qquad 0\le j\le k,
\end{equation}
for every $x\in K\setminus Z$ with $\dist(x,Z)<\rho_K$.  Then
$F\in C^k(U)$ and
\begin{equation}\label{eq:extension-vanishing}
 D^jF=0\quad\text{on }Z,
 \qquad 0\le j\le k.
\end{equation}
\end{lemma}
\begin{proof}
There is nothing to prove if $Z=\emptyset$.  Suppose that
$Z\ne\emptyset$.  We prove the result by constructing the derivatives
of $F$ on $Z$ inductively.

Fix $z\in Z$ and a compact neighborhood $K\Subset U$ of $z$.  Since
$\dist(x,Z)\le |x-z|$, estimate \eqref{eq:extension-assumption} with
$j=0$ gives
\[
 \frac{|F(x)-F(z)|}{|x-z|}
 \le C_K\frac{\dist(x,Z)^\mu}{|x-z|}
 \le C_K|x-z|^{\mu-1}\longrightarrow0
\]
as $x\to z$ through $U\setminus Z$; the quotient is zero for
$x\in Z$.  Thus $F$ is Fr\'echet differentiable at $z$ and
$DF(z)=0$.  Moreover,
\[
 |DF(x)|\le C_K\dist(x,Z)^{\mu-1}\longrightarrow0
 \qquad\text{as }x\to Z.
\]
Consequently, the derivative obtained by setting $DF=0$ on $Z$ is
continuous, and $F\in C^1$ near $z$.

Suppose inductively that $1\le\ell<k$, that $F\in C^\ell$ near $z$,
and that $D^jF=0$ on $Z$ for $0\le j\le\ell$.  Then
\[
 \frac{|D^\ell F(x)-D^\ell F(z)|}{|x-z|}
 \le C_K\frac{\dist(x,Z)^{\mu-\ell}}{|x-z|}
 \le C_K|x-z|^{\mu-\ell-1}\longrightarrow0,
\]
because $\mu>k\ge\ell+1$.  Hence $D^\ell F$ is differentiable at $z$
with derivative zero.  Finally,
\[
 |D^{\ell+1}F(x)|
 \le C_K\dist(x,Z)^{\mu-\ell-1}\longrightarrow0
 \qquad\text{as }x\to Z,
\]
so $D^{\ell+1}F$ is continuous across $Z$.  Induction completes the
proof.
\end{proof}

\begin{remark}
The conclusion is local.  It is enough that
\eqref{eq:extension-assumption} hold in a neighborhood of $Z$, since
$F$ is already smooth on $U\setminus Z$.
\end{remark}

\section{The intrinsic noncritical scale}
\label{sec:intrinsic}

We first obtain uniform higher-order estimates in terms of the size of
the gradient.  Fix
\begin{equation}\label{eq:compact-sets}
 K\Subset U\Subset\Omega
\end{equation}
and set
\begin{equation}\label{eq:H-def}
 H:=[Du]_{C^{0,\alpha}(U)}.
\end{equation}
If $H=0$, then $Du$ is constant on each connected component of $U$, and
all conclusions are immediate.  We therefore assume $H>0$.

We shall use the following standard noncritical regularity lemma.  Its
formulation also applies to the anisotropic fluxes considered in Section
\ref{sec:anisotropic}.  Set
\[
 \mathcal K_*:=\{\xi\in\R^d:1/2\le |\xi|\le3/2\}.
\]

\begin{lemma}[Noncritical smoothness and uniform Schauder bootstrap]
\label{lem:noncritical-smoothness}
\label{lem:schauder-bootstrap}
Let $L\ge1$ and $0<\alpha<1$.  Let $\mathcal O$ be an open
neighborhood of $\mathcal K_*$ and let
$\mathcal A\in C^{L+1}(\mathcal O;\R^d)$ satisfy
\begin{equation}\label{eq:general-flux-ellipticity}
 \lambda|\eta|^2
 \le \langle D_\xi\mathcal A(\xi)\eta,\eta\rangle,
 \qquad |D_\xi\mathcal A(\xi)|\le\Lambda
 \quad(\xi\in\mathcal K_*,\ \eta\in\R^d)
\end{equation}
for some $0<\lambda\le\Lambda$.  Suppose that
$v\in W^{1,2}(B_2)$ is a weak solution of
\begin{equation}\label{eq:general-flux-equation}
 \operatorname{div}\mathcal A(Dv)=g\quad\text{in }B_2
\end{equation}
and that $Dv$ has a $C^{0,\alpha}(B_2)$ representative satisfying
\begin{equation}\label{eq:bootstrap-hypotheses}
 \frac34\le |Dv|\le\frac54,\qquad
 \|Dv\|_{C^{0,\alpha}(B_2)}\le M,\qquad
 \|g\|_{C^{L-1,\alpha}(B_2)}\le G.
\end{equation}
Then $v\in C^{L+1,\alpha}(B_1)$ and
\begin{equation}\label{eq:bootstrap-conclusion}
 \|v\|_{C^{L+1,\alpha}(B_1)}
 \le C\bigl(d,\alpha,L,\lambda,\Lambda,
       \|\mathcal A\|_{C^{L+1}(\mathcal K_*)},
       M,G,\|v\|_{C^0(B_2)}\bigr),
\end{equation}
where $C$ is finite and increasing in its last three arguments.  If
$\mathcal A\in C^\infty(\mathcal O)$ and $g\in C^\infty(B_2)$, then
applying the estimate for every $L$ on smaller concentric balls gives
$v\in C^\infty(B_2)$.
\end{lemma}

\begin{proof}
We first justify differentiating the weak equation.  Fix concentric
balls $B''\Subset B'\Subset B_2$, and let $\psi$ solve
$\Delta\psi=g$ in $B'$ with zero boundary values.  Schauder estimates
on the ball give
\[
 \|\psi\|_{C^{2,\alpha}(B')}
 \le C\|g\|_{C^{0,\alpha}(B')}.
\]
Thus \eqref{eq:general-flux-equation} may be written
\[
 \operatorname{div}\bigl(\mathcal A(Dv)-D\psi\bigr)=0
 \qquad\text{in }B'.
\]
Let $\delta_h$ denote the difference quotient in the direction
$e_\ell$.  Subtracting the translated equations yields
\begin{equation}\label{eq:difference-quotient-equation}
 \operatorname{div}\bigl(a_h\delta_hDv\bigr)
 =\operatorname{div}\bigl(\delta_hD\psi\bigr),
\end{equation}
where
\[
 a_h(x):=\int_0^1D_\xi\mathcal A\bigl((1-t)Dv(x)
             +tDv(x+he_\ell)\bigr)\,dt.
\]
For $|h|$ sufficiently small, the H\"older continuity of $Dv$ and
the first two bounds in \eqref{eq:bootstrap-hypotheses} ensure that
every segment in this integral lies in $\mathcal K_*$.  Hence $a_h$
is uniformly bounded and elliptic, independently of $h$.  Testing
\eqref{eq:difference-quotient-equation} with
$\zeta^2\delta_hv$, where $\zeta$ is a cutoff for $B''\Subset B'$,
gives
\[
 \int_{B''}|\delta_hDv|^2
 \le C\int_{B'}|\delta_hv|^2
     +C\int_{B'}|\delta_hD\psi|^2
 \le C\int_{B'}|Dv|^2+C\int_{B'}|D^2\psi|^2.
\]
The difference-quotient characterization of Sobolev spaces therefore
gives $v\in W^{2,2}_{\loc}(B_2)$.

Set $w_\ell=\partial_\ell v$.  Taking the $\ell$-th weak derivative
of \eqref{eq:general-flux-equation} now gives
\begin{equation}\label{eq:linearized-divergence-equation}
 \operatorname{div}\bigl(D_\xi\mathcal A(Dv)Dw_\ell\bigr)
 =\partial_\ell g
 =\operatorname{div}(g e_\ell)
 \quad\text{weakly on smaller balls}.
\end{equation}
The coefficients are uniformly elliptic and belong to
$C^{0,\alpha}$, while $g e_\ell\in C^{0,\alpha}$.
The divergence-form Schauder estimate
\cite[Theorem~8.32]{GilbargTrudinger} yields
$w_\ell\in C^{1,\alpha}_{\loc}$ and hence
$v\in C^{2,\alpha}_{\loc}(B_2)$, with a quantitative estimate on
concentric balls.  Equation \eqref{eq:general-flux-equation} now holds
pointwise in nondivergence form:
\begin{equation}\label{eq:bootstrap-nondivergence}
 a_{ij}(Dv)v_{ij}=g,
 \qquad a_{ij}(\xi)=\partial_{\xi_j}\mathcal A_i(\xi),
 \qquad v_{ij}=\partial_{ij}v.
\end{equation}

The bounds in \eqref{eq:bootstrap-hypotheses} control the
$C^{0,\alpha}$ norm of $a_{ij}(Dv)$ and its ellipticity constants.
The interior nondivergence-form Schauder estimate
\cite[Theorem~6.2]{GilbargTrudinger} gives the quantitative
$C^{2,\alpha}$ bound on $B_{\rho_1}$, where
$\rho_s:=1+2^{-s}$.  If $1\le s\le L-1$, differentiating
\eqref{eq:bootstrap-nondivergence} $s$ times and applying the same
estimate on the next smaller ball gives inductively
\[
 \|v\|_{C^{s+2,\alpha}(B_{\rho_{s+1}})}
 \le C_s,
 \qquad 1\le s\le L-1.
\]
At the $s$-th step the new right-hand side contains $D^s g$ and
products of derivatives already controlled at earlier steps.  Thus
$g\in C^{L-1,\alpha}$ is exactly the finite regularity needed to reach
$C^{L+1,\alpha}$ on $B_1$.  This proves
\eqref{eq:bootstrap-conclusion}.
\end{proof}

The quantity $H$ in \eqref{eq:H-def} is the local
$C^{0,\alpha}$ seminorm of the gradient on the intermediate domain
$U$; it determines the intrinsic radius and is part of the constants
below.

\begin{proposition}[Intrinsic derivative estimates]
\label{prop:intrinsic}
Under the assumptions of Theorem \ref{thm:main}, there exist $q_0>0$
and $C_L<\infty$ such that
\begin{equation}\label{eq:u-derivative-estimate}
 |D^\ell u(x)|
 \le C_L|Du(x)|^{1-(\ell-1)/\alpha},
 \qquad 2\le\ell\le L+1,
\end{equation}
whenever $x\in K\setminus\crit(u)$ and $0<|Du(x)|\le q_0$.  The
constants may be chosen to depend only on
\[
 d,p,\alpha,L,\dist(K,\partial U),H,
 \quad\text{and}\quad \|f\|_{C^{L-1,\alpha}(U)}.
\]
\end{proposition}
\begin{proof}
Fix $x\in K\setminus\crit(u)$ and write
\[
 q:=|Du(x)|>0.
\]
Choose
\begin{equation}\label{eq:intrinsic-radius}
 r:=\left(\frac{q}{8H}\right)^{1/\alpha}.
\end{equation}
Put $\delta=\dist(K,\partial U)>0$.  It is enough to choose
\begin{equation}\label{eq:q0-choice}
 0<q_0\le
 \min\{1,8H,8H(\delta/2)^\alpha\}.
\end{equation}
Then $q\le q_0$ implies $r\le1$ and $2r\le\delta$, so
$B_{2r}(x)\Subset U$.  Since $0<\alpha<1$, for $z\in B_{2r}(x)$ we have
\[
 |Du(z)-Du(x)|
 \le H(2r)^\alpha
 \le\frac q4.
\]
Therefore
\begin{equation}\label{eq:gradient-comparability}
 \frac{3q}{4}\le |Du(z)|\le\frac{5q}{4}
 \qquad (z\in B_{2r}(x)).
\end{equation}

Define the normalized function
\begin{equation}\label{eq:v-def}
 v(y):=\frac{u(x+ry)-u(x)}{rq},
 \qquad y\in B_2.
\end{equation}
Then
\begin{equation}\label{eq:Dv-bounds}
 \frac34\le|Dv|\le\frac54
 \quad\text{in }B_2,
 \qquad
 [Dv]_{C^{0,\alpha}(B_2)}
 =\frac{r^\alpha}{q}[Du]_{C^{0,\alpha}(B_{2r}(x))}
 \le\frac18.
\end{equation}
The rescaled equation is
\begin{equation}\label{eq:v-divergence}
 \operatorname{div}\bigl(|Dv|^{p-2}Dv\bigr)=g(y),
 \qquad
 g(y):=rq^{1-p}f(x+ry).
\end{equation}
The rescaled right-hand side is uniformly controlled because of
\eqref{eq:alpha-choice}.  Indeed,
\begin{equation}\label{eq:forcing-prefactor}
 rq^{1-p}
 =(8H)^{-1/\alpha}q^{1/\alpha-(p-1)},
\end{equation}
which is bounded for $0<q\le1$.  For every integer $s\ge0$ and every
multi-index $\gamma$ of length $s$,
\begin{equation}\label{eq:g-derivatives}
 D_y^\gamma g(y)=r^{s+1}q^{1-p}D^\gamma f(x+ry).
\end{equation}
Because $1/\alpha-(p-1)\ge0$, for $0\le s\le L-1$,
\begin{align*}
 \|D_y^\gamma g\|_{L^\infty(B_2)}
 &\le C r^{s+1}q^{1-p}
  \le Cq^{(s+1)/\alpha-(p-1)}\le C,\\
 [D_y^\gamma g]_{C^{0,\alpha}(B_2)}
 &\le C r^{s+1+\alpha}q^{1-p}
  \le Cq^{(s+1+\alpha)/\alpha-(p-1)}\le C.
\end{align*}
Here and below $C$ may depend on $H$, $p$, $\alpha$, $L$, and
$\|f\|_{C^{L-1,\alpha}(U)}$, but it is independent of $x$ and $q$.
Consequently,
\begin{equation}\label{eq:g-uniform-holder}
 \|g\|_{C^{L-1,\alpha}(B_2)}\le C.
\end{equation}

Since $v(0)=0$ and $|Dv|\le5/4$, for every $y\in B_2$ we have
\begin{equation}\label{eq:v-c0-bound}
 |v(y)|=|v(y)-v(0)|
 \le |y|\sup_{B_2}|Dv|\le\frac52.
\end{equation}
Set $A_0(\xi)=|\xi|^{p-2}\xi$.  On $\mathcal K_*$ the map $A_0$ is
smooth and $D_\xi A_0$ is uniformly elliptic, with constants depending
only on $p$.  Since $v\in C^{1,\alpha}(B_2)$ and satisfies the rescaled
equation weakly, Lemma \ref{lem:schauder-bootstrap}, applied with
$\mathcal A=A_0$ and with \eqref{eq:Dv-bounds} and
\eqref{eq:g-uniform-holder}, yields
\begin{equation}\label{eq:v-high-estimate}
 \|v\|_{C^{L+1,\alpha}(B_1)}\le C_L.
\end{equation}
The constant is independent of $q$.  Finally, differentiating
\eqref{eq:v-def} gives
\[
 D^\ell_yv(0)=\frac{r^{\ell-1}}qD^\ell u(x).
\]
Combining this identity with \eqref{eq:v-high-estimate} and
\eqref{eq:intrinsic-radius}, we obtain
\[
 |D^\ell u(x)|
 \le C_Lqr^{1-\ell}
 \le C_Lq^{1-(\ell-1)/\alpha},
\]
which is \eqref{eq:u-derivative-estimate}.
\end{proof}

\begin{remark}[Finite regularity of the datum]
To prove Proposition \ref{prop:intrinsic} through order $L+1$, it is
enough to assume $f\in C^{L-1,\alpha}_{\loc}$.  The assumption
$f\in C^\infty$ is used only to state all orders simultaneously.
\end{remark}

\section{Homogeneous compositions and proof of the main theorem}
\label{sec:composition}

We now apply the intrinsic estimates to homogeneous functions of the
gradient.

\begin{proposition}[Derivative decay for homogeneous compositions]
\label{prop:composition}
Let $\Phi$ be as in Theorem \ref{thm:main}.  For every $L\in\N$ there
exist $q_0>0$ and $C_L<\infty$ such that
\begin{equation}\label{eq:composition-estimate}
 \left|D^j\bigl(\Phi(Du)\bigr)(x)\right|
 \le C_L|Du(x)|^{m-j/\alpha},
 \qquad 0\le j\le L,
\end{equation}
whenever $x\in K\setminus\crit(u)$ and
$0<|Du(x)|\le q_0$.
In addition to the dependencies in Proposition \ref{prop:intrinsic},
$C_L$ depends only on the $C^L$ norm of $\Phi$ on
$\{3/4\le|\xi|\le5/4\}$.
\end{proposition}

\begin{proof}
Let $q$, $r$, and $v$ be as in the proof of Proposition
\ref{prop:intrinsic}.  By homogeneity,
\begin{equation}\label{eq:composition-rescaling}
 \Phi(Du(x+ry))=\Phi(qDv(y))=q^m\Phi(Dv(y)).
\end{equation}
The estimate \eqref{eq:v-high-estimate}, the bounds
$3/4\le|Dv|\le5/4$, and the smoothness of $\Phi$ on the unit sphere
give the required uniform chain-rule bounds.  More explicitly, for
$j\ge1$, every term in the multivariate Fa\`a di Bruno expansion of
$D_y^j(\Phi(Dv))$ has the form
\begin{equation}\label{eq:faa-composition}
 D^s\Phi(Dv)
 \bigl[D^{\ell_1+1}v,\ldots,D^{\ell_s+1}v\bigr],
 \qquad
 \ell_i\ge1,\quad \ell_1+\cdots+\ell_s=j,
\end{equation}
for some $1\le s\le j$, with harmless combinatorial coefficients.
The derivatives of $\Phi$ are bounded on
$\{3/4\le|\xi|\le5/4\}$, and every derivative of $v$ occurring in
\eqref{eq:faa-composition} has order at most $j+1\le L+1$.
Consequently, \eqref{eq:v-high-estimate} implies
\begin{equation}\label{eq:normalized-composition-bound}
 \left|D_y^j\bigl(\Phi(Dv)\bigr)(0)\right|\le C_L,
 \qquad 0\le j\le L.
\end{equation}
On the other hand, differentiating \eqref{eq:composition-rescaling}
$j$ times at $y=0$ gives the exact scaling identity
\[
 r^jD_x^j\bigl(\Phi(Du)\bigr)(x)
 =q^mD_y^j\bigl(\Phi(Dv)\bigr)(0).
\]
Therefore
\[
 \left|D_x^j\bigl(\Phi(Du)\bigr)(x)\right|
 \le C_Lq^mr^{-j}
 =C_L(8H)^{j/\alpha}q^{m-j/\alpha}.
\]
Absorbing the fixed factor into $C_L$ proves
\eqref{eq:composition-estimate}.
\end{proof}

\begin{proof}[Proof of Theorem \ref{thm:main}]
Set
\[
 F(x):=\Phi(Du(x)),
 \qquad F=0\text{ on }\crit(u).
\]
Since $m>0$ and $\Phi$ is continuous on the unit sphere,
\[
 |\Phi(\xi)|\le
 \left(\sup_{|\theta|=1}|\Phi(\theta)|\right)|\xi|^m
 \qquad(\xi\ne0).
\]
Thus the definition at the critical set agrees with continuity of $Du$,
and $F\in C(\Omega)$.  To verify smoothness away from $\crit(u)$, fix
$x_*\notin\crit(u)$ and put $q_*=|Du(x_*)|$.  By continuity of $Du$,
there is $\rho>0$ such that $B_{2\rho}(x_*)\Subset\Omega$ and
$3q_*/4\le|Du|\le5q_*/4$ there.  The normalization
\[
 \widetilde v(y)=\frac{u(x_*+\rho y)-u(x_*)}{\rho q_*}
\]
satisfies the hypotheses of Lemma \ref{lem:noncritical-smoothness},
with flux $A_0(\xi)=|\xi|^{p-2}\xi$ and smooth right-hand side
$\rho q_*^{1-p}f(x_*+\rho y)$.  Applying that lemma for every $L$
shows that $u$, and therefore $F$, is smooth in a neighborhood of
$x_*$.

Fix $K\Subset U\Subset\Omega$, use the notation of
\eqref{eq:H-def}, and put
$\delta=\dist(K,\partial U)$.  If
$x\in K\setminus\crit(u)$ and
$\dist(x,\crit(u))<\delta/2$, a minimizing sequence for the distance
from $x$ to $\crit(u)$ remains in $U$.  Since $Du=0$ on
$\crit(u)$, H\"older continuity and passage to the infimum give
\begin{equation}\label{eq:q-vs-distance}
 |Du(x)|\le H\dist(x,\crit(u))^\alpha.
\end{equation}
In particular, after restricting to
$\dist(x,\crit(u))<\rho_K$ with $\rho_K>0$ sufficiently small, we have
$|Du(x)|\le q_0$, where $q_0$ is the constant in Proposition
\ref{prop:composition}.

For $0\le j\le k$, Proposition \ref{prop:composition} and
\eqref{eq:q-vs-distance} imply
\begin{equation}\label{eq:flatness-composition}
 |D^jF(x)|
 \le C|Du(x)|^{m-j/\alpha}
 \le C\dist(x,\crit(u))^{\alpha m-j}.
\end{equation}
Here the exponent $m-j/\alpha$ is positive because
$m>k/\alpha$ and $j\le k$.  Since $\alpha m>k$, estimate
\eqref{eq:flatness-composition} verifies the hypotheses of Lemma
\ref{lem:extension} with $Z=\crit(u)$ and $\mu=\alpha m$.  The lemma
therefore gives $F\in C^k_{\loc}(\Omega)$ and
$D^jF=0$ on $\crit(u)$ for every $0\le j\le k$.  The quantitative
estimate follows from Proposition \ref{prop:composition} with $L=k$.
\end{proof}

\begin{proof}[Proof of Corollary \ref{cor:power}]
Take $\Phi(\xi)=|\xi|^m$, which is smooth on
$\R^d\setminus\{0\}$ and positively homogeneous of degree $m$.
Theorem \ref{thm:main} gives \eqref{eq:power-ck}.  For $k=1$, choose
\[
 \alpha=\min\left\{\alpha_0,\frac1{p-1}\right\}.
\]
This gives \eqref{eq:power-c1}.  Formula
\eqref{eq:power-derivative} follows from the classical chain rule away
from $\crit(u)$ and from the vanishing statement on $\crit(u)$.
\end{proof}

\section{Autonomous anisotropic operators}
\label{sec:anisotropic}

The proof uses homogeneity and ellipticity on annuli, not the radial
form of the $p$-Laplace vector field.  We formulate the corresponding
extension.

Let $A\in C^\infty(\R^d\setminus\{0\};\R^d)$ satisfy
\begin{equation}\label{eq:A-homogeneous}
 A(t\xi)=t^{p-1}A(\xi),
 \qquad t>0, \xi\ne0.
\end{equation}
Since $p-1>0$, we extend $A$ continuously to the origin by setting
$A(0)=0$.
Assume that there are constants $0<\lambda\le\Lambda<\infty$ such that
for $|\xi|=1$ and $\eta\in\R^d$,
\begin{equation}\label{eq:A-ellipticity}
 \lambda|\eta|^2
 \le \langle D_\xi A(\xi)\eta,\eta\rangle,
 \qquad
 |D_\xi A(\xi)|\le\Lambda.
\end{equation}

\begin{theorem}[Anisotropic homogeneous equations]
\label{thm:anisotropic}
Let $k\in\N$ and let $u\in W^{1,p}_{\loc}(\Omega)$ solve
\begin{equation}\label{eq:A-equation}
 \operatorname{div}A(Du)=f
 \qquad\text{in }\Omega,
\end{equation}
where $f\in C^\infty(\Omega)$.  Suppose that
$Du\in C^{0,\alpha_0}_{\loc}(\Omega)$ and choose $\alpha$ as in
\eqref{eq:alpha-choice}.  If
$\Phi\in C^\infty(\R^d\setminus\{0\};\R^N)$ is positively homogeneous
of degree $m$ and $m>k/\alpha$, then
\[
 \Phi(Du)\in C^k_{\loc}(\Omega;\R^N).
\]
The conclusions \eqref{eq:vanishing-on-critical-set} and
\eqref{eq:composition-quantitative-intro} also hold.
\end{theorem}

\begin{proof}
Fix $K\Subset U\Subset\Omega$ and let
$H=[Du]_{C^{0,\alpha}(U)}$.  For
$H=0$ the gradient is constant on each connected component and the
conclusion is immediate, so assume $H>0$.  For
$x\in K\setminus\crit(u)$ with $q=|Du(x)|$ small, choose $r$ and $v$
as in \eqref{eq:intrinsic-radius} and \eqref{eq:v-def}.  The argument
in \eqref{eq:gradient-comparability}--\eqref{eq:Dv-bounds} uses only
the H\"older continuity of $Du$ and therefore still gives
\[
 \frac34\le |Dv|\le\frac54,
 \qquad [Dv]_{C^{0,\alpha}(B_2)}\le\frac18.
\]
Since $Du(x+ry)=qDv(y)$, the $(p-1)$-homogeneity of $A$ gives
$A(Du(x+ry))=q^{p-1}A(Dv(y))$.  Changing variables in the weak
formulation consequently yields
\begin{equation}\label{eq:A-rescaled}
 \operatorname{div}A(Dv)=rq^{1-p}f(x+ry).
\end{equation}

Differentiating \eqref{eq:A-homogeneous} shows that
$D_\xi A(t\xi)=t^{p-2}D_\xi A(\xi)$.  Hence
\eqref{eq:A-ellipticity}, compactness of the unit sphere, and smoothness
of $A$ imply uniform ellipticity and uniform bounds for all derivatives
of $D_\xi A$ on $\mathcal K_*$.  Lemma
\ref{lem:schauder-bootstrap}, now with $\mathcal A=A$, applies
directly to \eqref{eq:A-rescaled}, giving both noncritical smoothness
and the quantitative bootstrap.  Its right-hand side is exactly the
function $g$ estimated in
\eqref{eq:forcing-prefactor}--\eqref{eq:g-uniform-holder}.  We obtain
\begin{equation}\label{eq:anisotropic-v-bound}
 \|v\|_{C^{L+1,\alpha}(B_1)}\le C_L
 \qquad\text{for every }L\ge1,
\end{equation}
with constants independent of $x$ and $q$.

The scaling identity
$D_y^\ell v(0)=r^{\ell-1}q^{-1}D_x^\ell u(x)$ now gives
\eqref{eq:u-derivative-estimate}.  Homogeneity of $\Phi$ gives
\eqref{eq:composition-rescaling}, so the proof of Proposition
\ref{prop:composition} yields
\[
 |D^j(\Phi(Du))(x)|\le C_Lq^{m-j/\alpha}.
\]
Finally, the distance-to-the-critical-set argument
\eqref{eq:q-vs-distance}--\eqref{eq:flatness-composition} and Lemma
\ref{lem:extension} prove the asserted $C^k$ extension and vanishing of
all derivatives of order at most $k$ on $\crit(u)$.
\end{proof}

\begin{remark}
Theorem \ref{thm:anisotropic} assumes the gradient Hölder estimate rather
than reproving it.  Classical $C^{1,\alpha_0}$ theories apply to broad
classes of vector fields satisfying the corresponding $p$-growth,
monotonicity, and smoothness hypotheses.
\end{remark}

\section{The vectorial \texorpdfstring{$p$}{p}-Laplace system}
\label{sec:systems}

We next consider vector-valued solutions.  The main additional point
is the strong ellipticity of the linearized system.

Let $M\ge2$, write
$U=(U^1,\ldots,U^M):\Omega\to\R^M$, and identify its gradient
$DU=(\partial_iU^\beta)$ with an element of $\R^{M\times d}$ endowed
with the Frobenius norm.  We consider
\begin{equation}\label{eq:p-system}
 \operatorname{div}\bigl(|DU|^{p-2}DU\bigr)=F
 \qquad\text{in }\Omega,
\end{equation}
where $F:\Omega\to\R^M$ and the divergence is taken row by row.  Thus,
in components,
\[
 \partial_i\bigl(|DU|^{p-2}\partial_iU^\beta\bigr)=F^\beta,
 \qquad 1\le\beta\le M.
\]
The critical set is
\begin{equation}\label{eq:system-critical-set}
 \crit(U):=\{x\in\Omega:DU(x)=0\}.
\end{equation}

\begin{theorem}[Vectorial homogeneous compositions]
\label{thm:system}
Let $k\in\N$, let $F\in C^\infty(\Omega;\R^M)$, and let
$U\in W^{1,p}_{\loc}(\Omega;\R^M)$ be a weak solution of
\eqref{eq:p-system}.  Assume that
\begin{equation}\label{eq:system-holder}
 DU\in C^{0,\alpha_0}_{\loc}
       (\Omega;\R^{M\times d})
\end{equation}
for some $\alpha_0\in(0,1)$.  Choose
\begin{equation}\label{eq:system-alpha}
 0<\alpha<1,
 \qquad
 \alpha\le\min\left\{\alpha_0,\frac1{p-1}\right\}.
\end{equation}
Let
\[
 \Phi\in C^\infty
 \bigl(\R^{M\times d}\setminus\{0\};\R^N\bigr)
\]
be positively homogeneous of degree $m>0$, with respect to positive
scalar multiplication of matrices, and set $\Phi(0)=0$.  If
\begin{equation}\label{eq:system-threshold}
 m>\frac{k}{\alpha},
\end{equation}
then
\[
 \Phi(DU)\in C^k_{\loc}(\Omega;\R^N).
\]
Moreover,
\begin{equation}\label{eq:system-vanishing}
 D^j\bigl(\Phi(DU)\bigr)=0
 \quad\text{on }\crit(U),
 \qquad 0\le j\le k.
\end{equation}
For every $K\Subset\Omega$, there exist $q_0>0$ and $C<\infty$ such
that
\begin{equation}\label{eq:system-composition-estimate}
 \left|D^j\bigl(\Phi(DU)\bigr)(x)\right|
 \le C|DU(x)|^{m-j/\alpha},
 \qquad 0\le j\le k,
\end{equation}
whenever $x\in K\setminus\crit(U)$ and
$0<|DU(x)|\le q_0$.
\end{theorem}

\begin{proof}
Fix $K\Subset\mathcal U\Subset\Omega$ and put
\[
 H=[DU]_{C^{0,\alpha}(\mathcal U)}.
\]
The case $H=0$ is immediate.  Let
$x\in K\setminus\crit(U)$, set $q=|DU(x)|$, and, for $q$ sufficiently
small, define
\begin{equation}\label{eq:system-scaling}
 r=\left(\frac{q}{8H}\right)^{1/\alpha},
 \qquad
 V(y)=\frac{U(x+ry)-U(x)}{rq},
 \qquad y\in B_2.
\end{equation}
Exactly as in \eqref{eq:gradient-comparability} and
\eqref{eq:Dv-bounds}, after decreasing $q_0$ in terms of
$\dist(K,\partial\mathcal U)$, we have
\begin{equation}\label{eq:system-annulus}
 \frac34\le |DV|\le\frac54
 \quad\text{in }B_2,
 \qquad
 [DV]_{C^{0,\alpha}(B_2)}\le\frac18.
\end{equation}
The rescaled system is
\begin{equation}\label{eq:system-rescaled}
 \operatorname{div}\bigl(|DV|^{p-2}DV\bigr)=G(y),
 \qquad
 G(y)=rq^{1-p}F(x+ry).
\end{equation}
The restriction $\alpha\le1/(p-1)$ and the computation
\eqref{eq:forcing-prefactor}--\eqref{eq:g-uniform-holder} give, for
every $L\ge1$,
\begin{equation}\label{eq:system-G-bound}
 \|G\|_{C^{L-1,\alpha}(B_2)}\le C_L,
\end{equation}
uniformly in $x$ and $q$.  Since $V(0)=0$ and
$|DV|\le5/4$, also
\begin{equation}\label{eq:system-V-c0}
 \|V\|_{C^0(B_2)}\le\frac52.
\end{equation}

We verify the ellipticity needed for the vectorial bootstrap.  Set
\[
 \mathbb A(P)=|P|^{p-2}P,
 \qquad P\in\R^{M\times d}\setminus\{0\}.
\]
For $P,Q\in\R^{M\times d}$ with $P\ne0$,
\begin{equation}\label{eq:system-flux-derivative}
 D\mathbb A(P)[Q]
 =|P|^{p-2}Q
 +(p-2)|P|^{p-4}(P:Q)P,
\end{equation}
and hence
\begin{align}
 D\mathbb A(P)[Q]:Q
 &=|P|^{p-2}|Q|^2
 +(p-2)|P|^{p-4}(P:Q)^2 \notag\\
 &\ge \min\{1,p-1\}|P|^{p-2}|Q|^2,
 \label{eq:system-strong-ellipticity}
\end{align}
while the same expression is bounded from above by
$\max\{1,p-1\}|P|^{p-2}|Q|^2$.  Thus $D\mathbb A$ is uniformly
strongly elliptic and smooth on the fixed matrix annulus
\[
 \mathbb K_*=\{P\in\R^{M\times d}:1/2\le|P|\le3/2\}.
\]

We first justify the uniform Schauder estimates.  Difference quotients
of \eqref{eq:system-rescaled} satisfy, on the corresponding
$h$-shrunk ball,
\[
 \operatorname{div}\bigl(\mathbb B_h\delta_hDV\bigr)=\delta_hG,
\]
where
\begin{equation}\label{eq:system-difference-coefficient}
 \mathbb B_h(y)
 =\int_0^1D\mathbb A\bigl((1-t)DV(y)
              +tDV(y+he_\ell)\bigr)\,dt.
\end{equation}
For $|h|$ sufficiently small, all matrices on the segment in
\eqref{eq:system-difference-coefficient} lie in $\mathbb K_*$ by
\eqref{eq:system-annulus}.  Testing this system with a cutoff times
$\delta_hV$, using \eqref{eq:system-strong-ellipticity}, and retaining
the term containing $\delta_hG$, we obtain
\[
 \sup_{0<|h|<h_0}
 \|\delta_hDV\|_{L^2(B')}\le C
 \qquad\text{for every }B'\Subset B_2.
\]
Here \eqref{eq:system-G-bound}, with $L=2$, is used to control
$\|\delta_hG\|_{L^2}$.  Therefore
$V\in W^{2,2}_{\loc}(B_2;\R^M)$.  Each
$W_\ell=\partial_\ell V$ then solves a linear divergence-form,
strongly elliptic system with coefficients
$D\mathbb A(DV)\in C^{0,\alpha}$.  Its right-hand side is
$\partial_\ell G=\operatorname{div}(G e_\ell)$.  The interior Schauder
estimate for elliptic systems
\cite[Chapter~5]{GiaquintaMartinazzi2012} gives
$W_\ell\in C^{1,\alpha}_{\loc}$, and hence
$V\in C^{2,\alpha}_{\loc}$.

The system now holds pointwise in the form
\begin{equation}\label{eq:system-nondivergence}
 \partial_{P_i^\gamma}\mathbb A_j^\beta(DV)
 \,\partial_{ij}V^\gamma=G^\beta.
\end{equation}
We may now differentiate either \eqref{eq:system-rescaled} or
\eqref{eq:system-nondivergence}.  At
each stage the highest derivative solves a linear strongly elliptic
system with principal tensor $D\mathbb A(DV)$; all remaining terms are
products of already controlled derivatives of $V$, derivatives of
$G$, and derivatives of $\mathbb A$ on $\mathbb K_*$.  The H\"older
product and composition estimates therefore control the new
right-hand side.  Applying the interior Schauder estimate on a
strictly smaller concentric ball and iterating gives
\begin{equation}\label{eq:system-high-estimate}
 \|V\|_{C^{L+1,\alpha}(B_1)}\le C_L.
\end{equation}
The constants depend on $d,M,p,\alpha,L$, the local H\"older norm of
$DU$, the interior distance, and the indicated norm of $F$, but not
on $x$ or $q$.

Differentiating \eqref{eq:system-scaling} and using
\eqref{eq:system-high-estimate} gives
\begin{equation}\label{eq:system-U-derivatives}
 |D^\ell U(x)|
 \le C_Lq^{1-(\ell-1)/\alpha},
 \qquad 2\le\ell\le L+1.
\end{equation}
The homogeneity of $\Phi$ gives the exact identity
\[
 \Phi(DU(x+ry))=q^m\Phi(DV(y)).
\]
The same multivariate chain-rule argument as in Proposition
\ref{prop:composition} consequently yields
\eqref{eq:system-composition-estimate}.

It remains to cross the critical set.  If
$x\in K\setminus\crit(U)$ is sufficiently close to $\crit(U)$, then
the Hölder continuity of $DU$ gives
\[
 |DU(x)|\le H\dist(x,\crit(U))^\alpha.
\]
Combining this with \eqref{eq:system-composition-estimate}, we obtain
\[
 \left|D^j\bigl(\Phi(DU)\bigr)(x)\right|
 \le C\dist(x,\crit(U))^{\alpha m-j},
 \qquad 0\le j\le k.
\]
Since $\alpha m>k$, Lemma \ref{lem:extension}, applied componentwise,
proves the $C^k$ extension and \eqref{eq:system-vanishing}.
\end{proof}

\begin{corollary}[Powers of the matrix gradient]
\label{cor:system-power}
Under the hypotheses of Theorem \ref{thm:system},
\[
 |DU|^m\in C^k_{\loc}(\Omega)
 \qquad\text{whenever}\qquad m>\frac{k}{\alpha}.
\]
In particular,
\[
 |DU|^m\in C^1_{\loc}(\Omega)
 \qquad\text{for every}\qquad
 m>\max\left\{\frac1{\alpha_0},p-1\right\}.
\]
For $1\le i\le d$ its first derivative is
\begin{equation}\label{eq:system-power-derivative}
 \partial_i(|DU|^m)=
 \begin{cases}
 m|DU|^{m-2}
 \displaystyle\sum_{\beta=1}^M\sum_{j=1}^d
 (\partial_jU^\beta)(\partial_{ij}U^\beta),
 &DU\ne0,\\[8pt]
 0,&DU=0.
 \end{cases}
\end{equation}
\end{corollary}

\begin{remark}[Sharpness in the system class]
\label{rem:system-sharpness}
The scalar constant-source example embeds into the system by taking
\[
 U(x)=\left(\frac{p-1}{p}|x_1|^{p/(p-1)},0,\ldots,0\right),
 \qquad
 F=(1,0,\ldots,0).
\]
Then \eqref{eq:p-system} holds and
$|DU|^m=|x_1|^{m/(p-1)}$.  Hence the strict condition $m>p-1$ for a
universal $C^1$ result remains necessary even within the vectorial
class.
\end{remark}

\begin{remark}[Structural scope]
Theorem \ref{thm:system} concerns the Uhlenbeck-type radial flux
$\mathbb A(P)=|P|^{p-2}P$ and assumes the full Hölder continuity of
$DU$.  It is not a statement for arbitrary systems with merely
$p$-growth, for which full gradient regularity can fail.  The same
proof does extend to autonomous matrix-valued fluxes
$\mathbb A:\R^{M\times d}\to\R^{M\times d}$ that are
$(p-1)$-homogeneous, smooth off the origin, and uniformly strongly
elliptic on the unit matrix sphere, provided the solution has the
corresponding $C^{1,\alpha_0}$ regularity.
\end{remark}

\section{The parabolic \texorpdfstring{$p$}{p}-Laplace system}
\label{sec:parabolic-system}

Let $M\ge1$, let $\Omega_T=\Omega\times(0,T)$, and consider a weak
solution of
\begin{equation}\label{eq:parabolic-system}
 \partial_tU-\operatorname{div}\bigl(|DU|^{p-2}DU\bigr)=F
 \qquad\text{in }\Omega_T,
\end{equation}
where $U:\Omega_T\to\R^M$, $F:\Omega_T\to\R^M$, and $DU$ denotes the
spatial matrix gradient.  Its critical set is
\begin{equation}\label{eq:parabolic-critical-set}
 \crit_T(U):=\{(x,t)\in\Omega_T:DU(x,t)=0\}.
\end{equation}

By a local weak solution we mean
\[
 U\in C^0_{\loc}\bigl((0,T);L^2_{\loc}(\Omega;\R^M)\bigr)
 \cap L^p_{\loc}\bigl((0,T);W^{1,p}_{\loc}(\Omega;\R^M)\bigr)
\]
such that
\begin{equation}\label{eq:parabolic-weak-formulation}
 -\int_{\Omega_T}U\cdot\partial_t\varphi
 +\int_{\Omega_T}|DU|^{p-2}DU:D\varphi
 =\int_{\Omega_T}F\cdot\varphi
\end{equation}
for every $\varphi\in C_c^\infty(\Omega_T;\R^M)$.

The intrinsic scales in space and time are different.  We assume that
there are fixed
$\alpha,\beta\in(0,1)$ such that, on every compact cylinder
$\mathcal U_T\Subset\Omega_T$, there is a constant
$H=H(\mathcal U_T)$ for which
\begin{equation}\label{eq:parabolic-gradient-holder}
 |DU(x,t)-DU(y,s)|
 \le H\bigl(|x-y|^\alpha+|t-s|^\beta\bigr).
\end{equation}
For the forcing term to remain uniformly bounded after normalization,
we impose
\begin{equation}\label{eq:parabolic-alpha-restriction}
 \alpha\le\frac1{p-1}.
\end{equation}
Define
\begin{equation}\label{eq:parabolic-exponents}
 \gamma:=
 \max\left\{\frac1\alpha,
       \frac{\beta^{-1}+p-2}{2}\right\},
 \qquad
 \eta:=2\gamma+2-p.
\end{equation}
Then
\begin{equation}\label{eq:parabolic-exponent-relations}
 \gamma\ge p-1,\qquad
 \alpha\gamma\ge1,\qquad
 \beta\eta\ge1,\qquad
 \eta-\gamma=\gamma+2-p\ge1.
\end{equation}

\begin{theorem}[Parabolic homogeneous compositions]
\label{thm:parabolic-system}
Let $F\in C^\infty(\Omega_T;\R^M)$ and let $U$ be a weak solution of
\eqref{eq:parabolic-system}.  Assume that $DU$ satisfies
\eqref{eq:parabolic-gradient-holder} locally, where
\eqref{eq:parabolic-alpha-restriction} holds, and let $\gamma,\eta$ be
given by \eqref{eq:parabolic-exponents}.  Let
\[
 \Phi\in C^\infty
 \bigl(\R^{M\times d}\setminus\{0\};\R^N\bigr)
\]
be positively homogeneous of degree $m>0$, and set $\Phi(0)=0$.

For every pair of nonnegative integers $(a,b)$ and every compact
$K\Subset\Omega_T$, there exist $q_0>0$ and $C_{a,b}<\infty$ such that
\begin{equation}\label{eq:parabolic-mixed-estimate}
 \left|D_x^a\partial_t^b\bigl(\Phi(DU)\bigr)(z)\right|
 \le C_{a,b}|DU(z)|^{m-a\gamma-b\eta}
\end{equation}
whenever $z\in K\setminus\crit_T(U)$,
$0<|DU(z)|\le q_0$, and the exponent on the right is positive.
Here $D_x^a$ denotes any spatial derivative of total order $a$.
The constants depend only on the compact inclusions, the displayed
H\"older modulus, $d,M,p,\alpha,\beta,a,b$, a finite collection of
local derivatives of $F$, and a finite collection of derivatives of
$\Phi$ on a fixed matrix annulus.

Consequently, if $k\in\N$, then:
\begin{enumerate}
\item If $m>k\gamma$, all spatial derivatives
$D_x^a(\Phi(DU))$, $0\le a\le k$, extend continuously to
$\Omega_T$, vanish on $\crit_T(U)$, and
\begin{equation}\label{eq:parabolic-spatial-conclusion}
 \Phi(DU)\in
 C^0_{\loc}\bigl((0,T);C^k_{\loc}(\Omega;\R^N)\bigr).
\end{equation}
\item If $m>k\eta$, then
\begin{equation}\label{eq:parabolic-full-conclusion}
 \Phi(DU)\in C^k_{\loc}(\Omega_T;\R^N)
\end{equation}
in the ordinary space-time sense, and
\begin{equation}\label{eq:parabolic-mixed-vanishing}
 D_x^a\partial_t^b\bigl(\Phi(DU)\bigr)=0
 \quad\text{on }\crit_T(U)
 \qquad(a+b\le k).
\end{equation}
\end{enumerate}
\end{theorem}
\begin{proof}
Fix a compact cylinder $K\Subset\mathcal U_T\Subset\Omega_T$ on which
\eqref{eq:parabolic-gradient-holder} holds with constant $H$.  Let
$z_0=(x_0,t_0)\in K\setminus\crit_T(U)$ and set
\[
 q=|DU(z_0)|\in(0,1].
\]
Choose a fixed $c=c(H,\alpha,\beta)\in(0,1)$ sufficiently small and
put
\begin{equation}\label{eq:parabolic-intrinsic-scales}
 r=cq^\gamma,
 \qquad
 \vartheta=r^2q^{2-p}=c^2q^\eta.
\end{equation}
After decreasing $q_0$ in terms of the parabolic distance from $K$ to
$\partial\mathcal U_T$, the cylinder
\[
 Q_{2r,4\vartheta}(z_0)
 :=B_{2r}(x_0)\times(t_0-4\vartheta,t_0+4\vartheta)
\]
is compactly contained in $\mathcal U_T$.  By
\eqref{eq:parabolic-gradient-holder} and
\eqref{eq:parabolic-exponent-relations},
\begin{align*}
 |DU(x,t)-DU(z_0)|
 &\le H\bigl((2r)^\alpha+(4\vartheta)^\beta\bigr)\\
 &\le C H c^{\min\{\alpha,2\beta\}}q
\end{align*}
in this cylinder.  Taking $c$ smaller if necessary gives
\begin{equation}\label{eq:parabolic-gradient-annulus}
 \frac{3q}{4}\le|DU|\le\frac{5q}{4}
 \quad\text{in }Q_{2r,4\vartheta}(z_0).
\end{equation}

For $\rho>0$, write
\[
 \mathcal Q_\rho:=B_\rho\times(-\rho^2,\rho^2).
\]
On $\mathcal Q_2$ define
\begin{equation}\label{eq:parabolic-normalization}
 V(y,\tau)
 :=\frac{U(x_0+ry,t_0+\vartheta\tau)-U(x_0,t_0)}
 {rq}.
\end{equation}
Then
\begin{equation}\label{eq:parabolic-normalized-gradient}
 \frac34\le|D_yV|\le\frac54
 \quad\text{in }\mathcal Q_2.
\end{equation}
The choice $\vartheta=r^2q^{2-p}$ gives the normalized system
\begin{equation}\label{eq:parabolic-normalized-system}
 \partial_\tau V
 -\operatorname{div}_y\bigl(|D_yV|^{p-2}D_yV\bigr)
 =G(y,\tau),
\end{equation}
where
\begin{equation}\label{eq:parabolic-normalized-force}
 G(y,\tau)
 =rq^{1-p}F(x_0+ry,t_0+\vartheta\tau).
\end{equation}
Since $\gamma\ge p-1$, all normalized derivatives of $G$ are uniformly
bounded.  More precisely,
\begin{equation}\label{eq:parabolic-force-derivatives}
 D_y^a\partial_\tau^bG
 =r^{a+1}\vartheta^bq^{1-p}
 (D_x^a\partial_t^bF)(x_0+ry,t_0+\vartheta\tau),
\end{equation}
and the power of $q$ in the prefactor is
\[
 (a+1)\gamma+b\eta+1-p\ge0.
\]

We now justify the uniform parabolic estimates.  By
\eqref{eq:parabolic-gradient-annulus}, the coefficient tensor obtained
by linearizing the flux is smooth and uniformly strongly elliptic,
with constants depending only on $p$; see
\eqref{eq:system-flux-derivative}--%
\eqref{eq:system-strong-ellipticity}.  Moreover,
\eqref{eq:parabolic-gradient-holder} and the choices of $r$ and
$\vartheta$ give a uniform Hölder bound for $D_yV$ on
$\mathcal Q_2$, with parabolic exponent
\[
 \bar\alpha:=\min\{\alpha,2\beta\}.
\]
Spatial difference quotients of
\eqref{eq:parabolic-normalized-system} first justify the equation for
$W_\ell=\partial_\ell V$:
\begin{equation}\label{eq:parabolic-linearized-system}
 \partial_\tau W_\ell
 -\operatorname{div}_y
 \bigl(D\mathbb A(D_yV)D_yW_\ell\bigr)
 =\partial_\ell G.
\end{equation}
Here $\|W_\ell\|_{L^\infty(\mathcal Q_2)}\le5/4$, so no bound for the
time-dependent additive constant in $V$ is needed.

We first apply the energy estimate to the spatial difference
quotients of \eqref{eq:parabolic-normalized-system}.  The
strong ellipticity in
\eqref{eq:system-strong-ellipticity} gives uniform local
$L^2$ bounds for their spatial gradients.  Passing to the limit
justifies \eqref{eq:parabolic-linearized-system} weakly.  Its
coefficient tensor
\[
 \mathbb B(y,\tau):=D\mathbb A(D_yV(y,\tau))
\]
is uniformly strongly parabolic and has a uniform parabolic
$C^{0,\bar\alpha}$ norm.  The divergence-form Schauder estimate for
linear parabolic systems therefore gives
$W_\ell\in C^{1+\bar\alpha,(1+\bar\alpha)/2}_{\loc}$.  It follows
that
\eqref{eq:parabolic-linearized-system} can be written in
nondivergence form with H\"older lower-order coefficients; the
nondivergence-form parabolic Schauder estimate yields two further
spatial derivatives and one time derivative.  Repeated differentiation
of the equation, always on a smaller concentric cylinder, completes
the bootstrap.  At each step the new right-hand side is a finite sum
of products of already controlled derivatives of $D_yV$, derivatives
of $G$, and derivatives of $\mathbb A$ on the fixed annulus.
Consequently,
\begin{equation}\label{eq:parabolic-uniform-bootstrap}
 \sup_{\mathcal Q_1}
 |D_y^a\partial_\tau^bD_yV|
 \le C_{a,b}
 \qquad\text{for all }a,b\ge0.
\end{equation}
The constants are independent of $z_0$ and $q$; see
\cite[Chapters~III--IV]{LadyzhenskayaSolonnikovUraltseva1968}.  Indeed,
the annulus condition \eqref{eq:parabolic-normalized-gradient}, the
uniform H\"older modulus of $D_yV$, and
\eqref{eq:parabolic-force-derivatives} ensure that every constant in
the iteration is independent of the base point and of $q$.

By homogeneity,
\begin{equation}\label{eq:parabolic-composition-scaling}
 \Phi\bigl(DU(x_0+ry,t_0+\vartheta\tau)\bigr)
 =q^m\Phi(D_yV(y,\tau)).
\end{equation}
Combining the chain rule,
\eqref{eq:parabolic-uniform-bootstrap}, and
\eqref{eq:parabolic-intrinsic-scales}, we obtain
\begin{align*}
 \left|D_x^a\partial_t^b
 \bigl(\Phi(DU)\bigr)(z_0)\right|
 &\le C_{a,b}q^mr^{-a}\vartheta^{-b}\\
 &\le C_{a,b}q^{m-a\gamma-b\eta},
\end{align*}
which proves \eqref{eq:parabolic-mixed-estimate}.

It remains to extend the derivatives across $\crit_T(U)$.  At
$z_*=(x_*,t_*)\in\crit_T(U)$, assumption
\eqref{eq:parabolic-gradient-holder} gives
\begin{align}
 |DU(x_*+h,t_*)|&\le H|h|^\alpha,
 \label{eq:parabolic-space-flatness}\\
 |DU(x_*,t_*+s)|&\le H|s|^\beta.
 \label{eq:parabolic-time-flatness}
\end{align}
Suppose $a+b<k$ and the mixed derivative
$D_x^a\partial_t^b(\Phi(DU))$ has already been extended by zero on the
critical set.  If the next derivative is spatial, then
\[
 m-a\gamma-b\eta>\gamma\ge\frac1\alpha;
\]
if it is temporal, then
\[
 m-a\gamma-b\eta>\eta\ge\frac1\beta.
\]
These inequalities follow from $m>k\eta$,
$\gamma\le\eta$, and $a+b+1\le k$.  Estimates
\eqref{eq:parabolic-mixed-estimate},
\eqref{eq:parabolic-space-flatness}, and
\eqref{eq:parabolic-time-flatness} show that the corresponding
difference quotient at $z_*$ tends to zero.  The estimate for the next
derivative also tends to zero as $z\to\crit_T(U)$, proving continuity.
Induction over coordinate derivatives proves
\eqref{eq:parabolic-full-conclusion} and
\eqref{eq:parabolic-mixed-vanishing}.

For the spatial conclusion, the same induction is performed only in
the $x$ variables.  The condition $m>k\gamma$ is then sufficient, and
joint continuity in $(x,t)$ follows from
\eqref{eq:parabolic-mixed-estimate} with $b=0$ and the smoothness away
from the critical set.
\end{proof}

\begin{corollary}[Powers of the parabolic matrix gradient]
\label{cor:parabolic-power}
Under the hypotheses of Theorem \ref{thm:parabolic-system},
\begin{align*}
 |DU|^m&\in
 C^0_{\loc}\bigl((0,T);C^k_{\loc}(\Omega)\bigr)
 &&\text{if }m>k\gamma,\\
 |DU|^m&\in C^k_{\loc}(\Omega_T)
 &&\text{if }m>k\eta.
\end{align*}
All corresponding spatial or mixed derivatives vanish on
$\crit_T(U)$.
\end{corollary}

\begin{remark}[Standard parabolic Hölder exponents]
\label{rem:standard-parabolic-exponents}
If \eqref{eq:parabolic-gradient-holder} holds with
$\beta=\alpha/2$, then
\begin{equation}\label{eq:standard-parabolic-gamma-eta}
 \gamma=\frac1\alpha+\frac{(p-2)_+}{2},
 \qquad
 \eta=\frac2\alpha+(2-p)_+.
\end{equation}
Thus spatial and temporal derivatives have different homogeneity
costs even though both arise from the same normalized uniformly
parabolic system.
\end{remark}

\begin{remark}[Sharpness]
The thresholds in Theorem \ref{thm:parabolic-system} are sufficient
and are not asserted to be optimal.  The stationary elliptic example
\eqref{eq:model-intro}, embedded in one component, is also a parabolic
example and retains the necessary spatial condition $m>p-1$ for
universal $C^1$ regularity.  Temporal behavior at an interior time can
be tested with
\[
 U(x,t)=(t-t_*)x_1e_1,\qquad F(x,t)=x_1e_1,
 \qquad t_*\in(0,T),
\]
for which \eqref{eq:parabolic-system} holds and
$|DU|^m=|t-t_*|^m$.  Hence time regularity also has independent
endpoint obstructions, while the general optimal mixed threshold
remains open.
\end{remark}

\section{A porous-medium counterpart}
\label{sec:porous-medium}

We finally discuss the porous medium equation.  Here the intrinsic
scale is determined by the value of the solution rather than by the
size of its gradient.  Let $\mu>1$ and let $u\ge0$ be a weak solution
of
\begin{equation}\label{eq:porous-medium}
 \partial_tu-\Delta(u^\mu)=f
 \qquad\text{in }\Omega_T,
\end{equation}
where $f\in C^\infty(\Omega_T)$.  Assume that there are
$\alpha,\beta\in(0,1)$ such that, locally in $\Omega_T$,
\begin{equation}\label{eq:porous-holder}
 |u(x,t)-u(y,s)|
 \le H\bigl(|x-y|^\alpha+|t-s|^\beta\bigr).
\end{equation}
Define
\begin{equation}\label{eq:porous-costs}
 \gamma_\mu:=
 \max\left\{\frac1\alpha,\frac{\mu}{2},
 \frac{\beta^{-1}+\mu-1}{2}\right\},
 \qquad
 \eta_\mu:=2\gamma_\mu+1-\mu,
 \qquad
 \Lambda_\mu:=\max\{\gamma_\mu,\eta_\mu\}.
\end{equation}
Then $\alpha\gamma_\mu\ge1$, $\beta\eta_\mu\ge1$, and
$2\gamma_\mu-\mu\ge0$.

\begin{proposition}[Regularizing powers for porous-medium flow]
\label{prop:porous-power}
Under the preceding assumptions, let $m>0$.  For every compact
$K\Subset\Omega_T$ and every pair of nonnegative integers $a,b$, there
are $q_0>0$ and $C_{a,b}<\infty$ such that
\begin{equation}\label{eq:porous-mixed-estimate}
 \bigl|D_x^a\partial_t^b(u^m)(z)\bigr|
 \le C_{a,b}u(z)^{m-a\gamma_\mu-b\eta_\mu}
\end{equation}
whenever $z\in K$, $0<u(z)\le q_0$, and the exponent on the
right-hand side is positive.  Consequently, for every $k\in\N$,
\begin{align}
 u^m&\in
 C^0_{\loc}\bigl((0,T);C^k_{\loc}(\Omega)\bigr)
 &&\text{if }m>k\gamma_\mu,
 \label{eq:porous-spatial-ck}\\
 u^m&\in C^k_{\loc}(\Omega_T)
 &&\text{if }m>k\Lambda_\mu.
 \label{eq:porous-full-ck}
\end{align}
All corresponding derivatives vanish on the zero set $\{u=0\}$.
\end{proposition}

\begin{proof}
Fix $z_0=(x_0,t_0)$ with
$q=u(z_0)\in(0,1]$ and choose
\[
 r=cq^{\gamma_\mu},
 \qquad
 \vartheta=r^2q^{1-\mu}=c^2q^{\eta_\mu},
 \qquad
 v(y,\tau)=q^{-1}u(x_0+ry,t_0+\vartheta\tau).
\]
By \eqref{eq:porous-holder} and the first two inequalities following
\eqref{eq:porous-costs}, $c$ can be fixed so that
$3/4\le v\le5/4$ on a fixed cylinder.  Equation
\eqref{eq:porous-medium} becomes
\[
 \partial_\tau v-\Delta_y(v^\mu)=G,
 \qquad
 G=r^2q^{-\mu}
 f(x_0+ry,t_0+\vartheta\tau).
\]
The derivatives of $G$ have prefactor
$r^{a+2}\vartheta^bq^{-\mu}$, whose power of $q$ is
$(a+2)\gamma_\mu+b\eta_\mu-\mu\ge0$.  Since $v$ remains in a fixed
positive interval, the normalized equation is uniformly parabolic,
and the classical interior Schauder bootstrap gives uniform bounds
for all mixed derivatives of $v$.  The identity
\[
 u^m(x_0+ry,t_0+\vartheta\tau)=q^mv(y,\tau)^m
\]
then yields \eqref{eq:porous-mixed-estimate}.  Finally,
\eqref{eq:porous-holder} controls $u$ by
$H|x-x_*|^\alpha$ or $H|t-t_*|^\beta$ at a point
$(x_*,t_*)\in\{u=0\}$.  The same coordinatewise extension argument
used in Theorem \ref{thm:parabolic-system} proves
\eqref{eq:porous-spatial-ck}--\eqref{eq:porous-full-ck}.
\end{proof}

\begin{remark}[Free-boundary obstruction]
The conclusion is consistent with the classical free-boundary theory;
see \cite{CaffarelliFriedman1980,DaskalopoulosHamilton1998,
Vazquez2007}.  At a regular interface point the Barenblatt solution
has the spatial behavior
\[
 u(x,t)\simeq \dist(x,\partial\{u(\cdot,t)>0\})_+^{1/(\mu-1)}.
\]
Thus $u^m$ behaves like the positive part of the distance raised to
$m/(\mu-1)$, and universal spatial $C^k$ regularity generally
requires the strict condition $m>k(\mu-1)$.  In particular, the
pressure power $u^{\mu-1}$ is typically Lipschitz, but it need not be
$C^1$ across the free boundary.
\end{remark}

\medskip
\noindent\textbf{Data availability.}
This research has no associated data.

\medskip
\noindent\textbf{Conflict of interest.}
The authors declare that they have no conflicts of interest.

\medskip
\noindent\textbf{Acknowledgments.}
Quoc-Hung Nguyen is supported by the CAS Project for Young Scientists
in Basic Research, Grant No.~YSBR-031, and by the NSFC under Grant
Nos.~1251101538 and 12595282.

\medskip
\noindent\textbf{AI assistance statement.}
The authors used the AI model GPT-5.6 Sol to assist with calculations
and writing.
The main ideas, mathematical validation, and all final checks remain
the sole responsibility of the authors.

\end{document}